\documentclass[11pt]{amsart}
\usepackage{amsmath,amsthm,amssymb,tikz}
\usepackage[applemac]{inputenc}

\usepackage{amsfonts,txfonts}
\theoremstyle{theorem}
\newtheorem{theorem}{Theorem}
\newtheorem{lemma}{Lemma}
\theoremstyle{remark}
\newtheorem{remark}{Remark}

\numberwithin{equation}{section}

\newcommand{\ve}{\varepsilon}

\newcommand{\deb}{\rightharpoonup}

\newcommand{\pical}{\mathcal{P}}

\newcommand{\T}{\mathrm{T}}

\newcommand{\R}{\mathbb{R}}

\newcommand\dd {\mathrm{d}}

\DeclareMathOperator{\argmin}{argmin}

\newcommand\blfootnote[1]{%
  \begingroup
  \renewcommand\thefootnote{}\footnote{#1}%
  \addtocounter{footnote}{-1}%
  \endgroup
}

\begin{document}
\author{Jos\'e-Antonio Carrillo, Filippo Santambrogio}
\title[$L^\infty$ estimates for the Keller-Segel JKO scheme]{$L^\infty$ estimates for the JKO scheme\\ in parabolic-elliptic Keller-Segel systems}
\address{{\bf J.-A. C.} Department of Mathematics, Imperial College London, South Kensington Campus, SW7 2AZ London UK \\ \newline \indent \indent
Email: {\tt carrillo@imperial.ac.uk} \\ \newline \indent 
 {\bf F. S.} Laboratoire de Math\'ematiques d'Orsay, Univ. Paris-Sud, CNRS, Universit\'e Paris-Saclay, 91405 Orsay Cedex, France\\ \newline \indent \indent Email: {\tt filippo.santambrogio@math.u-psud.fr}.}
\maketitle

\vspace{-1.0cm}

\begin{abstract} We prove $L^\infty$ estimates on the densities that are obtained via the JKO scheme for a general form of a parabolic-elliptic Keller-Segel type system, with arbitrary diffusion, arbitrary mass, and in arbitrary dimension. Of course, such an estimate blows up in finite time, a time proportional to the inverse of the initial $L^\infty$ norm. This estimate can be used to prove short-time well-posedness for a number of equations of this form regardless of the mass of the initial data. The time of existence of the constructed solutions coincides with the maximal time of existence of Lagrangian solutions without the diffusive term by characteristic methods.
\end{abstract}

\blfootnote{JAC was partially supported by the Royal Society via a Wolfson Research Merit Award and by EPSRC grant number EP/P031587/1. The work has been finished during a visit of FS to the Imperial College, in the framework of a joint CNRS-Imperial Fellowship; the hospitality and the financial support of the Imperial College are warmly acknowledged.}

\vspace{-0.5cm}

\section{Introduction}
We consider in this work a general version of the so-called parabolic-elliptic Keller-Segel system, i.e.
\begin{equation}\label{syst}
\begin{cases}\partial_t \rho+\chi\nabla\cdot(\rho\nabla u)-\nabla\cdot(\rho\nabla f'(\rho))=0&\mbox{ in }[0,T]\times \Omega,\\
			-\Delta u=\rho &\mbox{ in $ \Omega$ for every $t\in [0,T]$},\\
			u=0&\mbox{ on $ \partial\Omega$ for every $t\in [0,T]$},\\
			\rho(\nabla u - \nabla f'(\rho))\cdot \mathbf{n}=0&\mbox{ on $[0,T]\times \partial\Omega$},\\
			\rho(0,x)=\rho_0(x)& \mbox{ in } x\in \Omega .
			\end{cases}
			\end{equation}
This classical system \cite{KS,JL,HP,CC} models the evolution of a population $\rho$ of bacteria, which diffuse and are advected by a drift, being attracted by the high values of a chemoattractant substance whose concentration is given by $u$. The term $\nabla\cdot(\rho\nabla f'(\rho))$ is a linear diffusion $\Delta \rho$ whenever $f(s)= s\log s$, or acts as a nonlinear diffusion $\Delta \Psi(\rho)$ with diffusion coefficient $\Psi'(s)=s f''(s)$ in general. The chemoactractant density $u$ is ruled by the distribution of the bacteria itself since they produce themselves the chemical/protein to which they are attracted to as a way of cell signalling. The initial density $\rho_0$ is supposed to be a probability density, which is always possible up to a suitable scaling. The constant $\chi>0$ and the convex and superlinear function $f:\R_+\to\R$ are parameters of the model. Note that the fact that the values $u(t,\cdot)$ only depend on the values $\rho(t,\cdot)$ at the same time is a strong modeling assumption, which corresponds to the fact that adjustement in the distribution of the chemoattractant $u$ occur at a time scale much faster than the movement of the bacteria. This gives rise to the above parabolic-elliptic system, while different models involving $\partial_t u$ would be possible, and have been studied in, for instance \cite{CalCor,BCKKLL} and the references therein.

We concentrate, for simplicity, on the case where $\Omega$ is a bounded convex set, and we impose Dirichlet boundary conditions on $u$, together with no-flux boundary conditions on $\rho$ through $\partial\Omega$. The results of the paper can easily be adapted to the case where the Dirichlet conditions on $u$ are replaced by Neumann conditions by taking $-\Delta u=\rho-c$ (for $c=1/|\Omega|$) or $-\Delta u+u=\rho$, with $\nabla u\cdot \mathbf{n}=0$ on $\partial\Omega$. Also the whole space case $\Omega=\R^d$ can be treated. Both of these variants will be addressed at the end of Section \ref{2} but, for simplicity of the exposition, the main theorems will be stated in the Dirichlet case in bounded domains.

The PDE system \eqref{syst} is known to be the gradient-flow of the energy
$$
J(\rho):=\int f(\rho(x))\dd x-\frac{\chi}{2}\int |\nabla u(x)|^2\dd x,
$$
where $u$ depends on $\rho$ through -$\Delta u=\rho$ with $u=0$ on $ \partial\Omega$, with respect to the Wasserstein metric $W_2$. For the notion of gradient flows and of gradient flows for this metric, we refer to \cite{villani,AmbGigSav,GradFlowSurvey}.

This means, among other things, that a way to produce a solution to \eqref{syst} is to fix a time step $\tau>0$ and then iteratively solve the minimization problem
$$
\rho_{k+1}\in \argmin \left\{J(\rho)+\frac{W_2^2(\rho,\rho_k)}{\tau}\;:\; \rho\in\pical(\Omega) \right\},
$$
where the minimum is taken over all probability densities $\rho\in\pical(\Omega)$, which must be absolutely continuous because of the superlinear term $\int f(\rho)$. One defines a $W_2$-continuous curve in $\pical(\Omega)$ setting $\rho^\tau(t)=\rho_k$ for $t=k\tau$ and interpolating using Wasserstein geodesics in the intervals $(k\tau,(k+1)\tau)$. Then, it is possible to prove that the limit as $\tau\to 0$ provides a solution of \eqref{syst}. This discrete-in-time scheme to provide a solution is known under the name of Jordan-Kinderlehrer-Otto scheme. It was first introduced in \cite{JKO} for the linear Fokker-Planck equation, and then subsequently used for the porous medium equation \cite{Otto porous}.

Yet, in this very precise case some difficulties arise due to the minus sign in front of the term $\int |\nabla u|^2$. This energy, in terms of $\rho$, corresponds to a squared norm in the dual of $H^1_0(\Omega)$ and it is clear that semicontinuity and bounds from below may fail when it is taken with a negative sign. On the other hand, it is an energy of order $-1$, in the sense that it acts as an $L^2$ norm of antiderivatives of $\rho$, and one can expect that the term $\int f(\rho)$ can somehow compensate it. This turns out to depend on $f$, on the dimension, and on $\chi$. Indeed, in dimension $d=2$, in the case of linear diffusion, i.e. $f(t)=t\log t$, it is proven that the above minimization problem has a solution, and the JKO scheme converges, as soon as $\chi<8\pi$, see \cite{BlaDolPer}. In the critical case $\chi=8\pi$ solutions exist and they can also be constructed by variants of the JKO scheme, see \cite{BlaCarMas,BlaCarCar}. In higher dimension $d>2$, this critical mass phenomena exists corresponding to other critical energies, of the form $f(t)=t^{m}$ with $m=2(d-1)/d$, see \cite{BlaCarLau}.

In this paper we will consider a slightly modified JKO scheme, adding a constraint, depending on $\tau$ and disappearing as $\tau\to 0$, on the $L^\infty$ norm of $\rho$ in the above minimization problem. We then prove that a solution at each time step exists, and that it satisfies an $L^\infty$ bound that makes the constraint unbinding, and that we can iterate. Unfortunately, this estimate will explode in finite time $T$, with $T=(\chi ||\rho_0||_{L^\infty})^{-1}$. Yet, the good news is the fact that this bound will be completely independent of the choice of $f$ and of $\chi$. Essentially, as the reader can easily guess, we are providing $L^\infty$ bounds that hold for the system without diffusion, as in \cite{Poupaud,BerLauLeg}, which is compatible with an estimate exploding in time. The interest of our work mainly resides in its discrete-in-time nature, namely the fact that the estimate is proven along the JKO scheme. Let us finally mention that $L^\infty$ bounds at the level of the JKO scheme were recently shown for fractional porous medium equations \cite{LMS} with totally different goals and strategy.

Since uniqueness results for the Keller-Segel system exist under $L^\infty$ assumptions \cite{CarLisMai,LW} for certain nonlinearities, see \cite{AmbSer,Craig,Loeper} for related results, then this $L^\infty$ bound allows for existence (and uniqueness with the additional assumptions in \cite{CarLisMai,LW}) results in short time. We prove these existence results in Section \ref{3}, and that they are of interest because they are very general in terms of the diffusion nonlinearity. They provide short-time existence even above critical mass. In short, we show in this work that solutions of the Keller-Segel model exists, and they are unique under additional assumptions on the nonlinearities, at least for the time of existence and uniqueness of Lagrangian bounded solutions of the model without diffusion. 


\section{$L^\infty$-estimates}\label{2}

To make our strategy precise, we will consider the following minimization problem, for fixed $\tau>0$, and for $g\in L^\infty(\Omega)$:
\begin{equation}\label{min}
\min \left\{J(\rho)+\frac{W_2^2(\rho,g)}{\tau}\;:\; \rho\in\pical(\Omega),\,\rho\leq M:=\frac 1{\chi\tau}\right\}
\end{equation}
We will prove the following theorem
\begin{theorem}
Problem \eqref{min} admits at least one solution; moreover, if $\Omega$ is convex, $\inf_{t>0} tf''(t)>0$  and $\log g\in C^{0,\alpha}(\Omega)$, then for every $\lambda>1$ there exists a constant $c_0=c_0(\lambda,\chi,d)$ such that, if $\tau ||g||_{L^\infty}\leq c_0$, then any solution $\rho$ to this problem satisfies
$$||\rho||_{L^\infty}\leq ||g||_{L^\infty}\frac{1}{1-\lambda\tau\chi||g||_{L^\infty}},$$ 
or, equivalently
\begin{equation}\label{1X1Y}
||\rho||_{L^\infty}^{-1}\geq ||g||_{L^\infty}^{-1}-\lambda \tau\chi.
\end{equation}
Removing the assumptions $\inf_{t>0} tf''(t)>0$  and $\log g\in C^{0,\alpha}(\Omega)$, the same $L^\infty$ estimate is true for at least one particular solution of \eqref{min}.
\end{theorem}

\begin{proof}
First, we prove the existence of an optimizer $\rho$. We take a minimizing sequence $\rho_n\in \pical(\Omega)\cap L^\infty (\Omega)$ and thus, we can extract a weakly converging subsequence $\rho_n\deb\rho$ (we have weak-* convergence in $L^\infty$ and in the space of measures, and weak in all $L^p(\Omega)$, $1\leq p<\infty$). All the terms, except the potential energy, in the functional we minimize are classically known to be lower semicontinous for the weak convergence of probability measures. Since the potential energy term is the $H^1$ norm of $u_n$, solution of $-\Delta u_n=\rho_n$, modulo its sign, we make use of the $L^\infty$ bound $\rho_n$ to deduce that $\rho_n\to\rho$ in $H^{- 1}$ (using the compact injection of $L^2$ into $H^{-1}$), and hence the corresponding $u_n$ strongly converges in $H^1$ to the solution of $-\Delta u=\rho$, which makes this term continuous. This proves that the limit measure $\rho$ is a minimizer.

We now consider the optimality conditions in this minimization problem and we use the assumptions on $\Omega$, $f$ and $g$.  First, by the same arguments in \cite[Lemma 8.6]{OTAM}, we observe that the optimal $\rho$ has a density which is strictly positive a.e. (the assumption on $f''$ implies that $f$ behaves at $t=0$ at least as the entropy $f(t)=t\log t$, and in particular $f'(0)=-\infty$). Hence, the Kantorovich potential $\phi$ in the quadratic transport (for the cost $c(x,y)=|x-y|^2/2$) from $\rho$ to $g$ is unique, see \cite[Proposition 7.18]{OTAM}. Then, the optimality conditions for this constrained problem read as follows: there exists a constant $c$ and a continuous function $p$ such that 
\begin{equation}\label{condwithp}
f'(\rho)+p-\chi u+\frac{\phi}{\tau}=c,\quad p\geq 0, \, \rho\leq M,\, p(M-\rho)=0.
\end{equation}
The proof of this fact is quite classical in optimization under density constraints, see \cite{MRS10}: we first note that the first variation of the functional that we minimize is $h:=f'(\rho)-\chi u+\frac{\phi}{\tau}$ where $\phi$ is the Kantorovich potential such that the map $x-\nabla \phi$ gives the optimal map from $\rho$ to $g$, for this derivation we refer to \cite[Section 7.2.1, 7.2.2 and 7.2.3]{OTAM}. Then we observe that the optimality of $\rho$ implies 
$$
\int h\tilde\rho\geq \int h\rho
$$ 
for every other admissible $0\leq \tilde\rho\leq M$, which means that $\rho$ is equal to $M$ on a sublevel set $\{h<c\}$, equal to $0$ on $\{h>c\}$, and between $0$ and $M$ on the level set $\{h=c\}$. The result is then obtained by taking $p:=(c-h)_+$, see for instance \cite[Theorem 4]{BCLR2} or \cite[Lemma 3.2]{MRS10}.

From these conditions, we deduce some regularity for $\rho$. Indeed, distinguishing the cases $\rho=M$ and $\rho<M$, we get
$$
f'(\rho)=\min\left\{f'(M),c+\chi u-\frac\phi\tau\right\}.
$$
Since $\rho\in L^\infty(\Omega)$, we have $u\in W^{2,p}(\Omega)$ for large $p$, hence $u$ is Lipschitz continuous, and $\phi$ is also Lipschitz continuous. Hence, the same is true for $f'(\rho)$ and, using the lower bound on $ f''$, for $\rho$ itself. Moreover, we also obtain $\inf \rho>0$. Of course this regularity depends on $\tau$, but it allows us to perform some computations. 

We next want to estimate the $L^\infty$-norm of an optimal $\rho$. Using $\log g\in C^{0,\alpha}(\Omega)$, we are in a case where both $g$ and $\rho$ are $C^{0,\alpha}(\Omega)$ and bounded from below, and Caffarelli's theory (see \cite{caf1,caf2,caf3,DePFig survey}) implies $\phi\in C^{2,\alpha}(\Omega)$. Note that from $\rho\in C^{0,\alpha}(\Omega)$ one also obtains $u\in C^{2,\alpha}(\Omega)$. 

Now, take a point $x_0\in \Omega$ which is a minimum point for $\phi-\tau\chi u$. Such a point exists since $\Omega$ is compact and these functions are continuous. Yet, from \eqref{condwithp}, one can see that this point maximizes $f'(\rho)+p$, and hence $\rho$, so that we have $||\rho||_{L^\infty}=\rho(x_0)$. First, let us prove $x_0\notin\partial\Omega$. Indeed, if it were on the boundary, we would have $(\nabla\phi(x_0)-\tau\chi\nabla u(x_0))\cdot \mathbf{n}\leq 0$. Yet, the optimal transport map $T$ from $\rho$ to $g$ has the form $T(x)=x-\nabla\phi(x)$, and maps $\Omega$ onto $\Omega$. Hence $(T(x_0)-x_0)\cdot \mathbf{n}\leq 0$, but this means $\nabla\phi(x_0)\cdot \mathbf{n}\geq 0$. Moroever, $u=0$ on $\partial\Omega$ and $-\Delta u=\rho\geq 0$ in $\Omega$. The strong maximum principle provides $\nabla u(x_0)\cdot \mathbf{n}<0$ on every point of the boundary. Hence it is not possible to have $x_0\in\partial\Omega$. Once we know that the minimizer of $\phi-\tau\chi u$ lies in the interior of $\Omega$, we deduce $\Delta\phi(x_0)-\tau\chi\Delta u(x_0)\geq 0$, but this means $\Delta\phi(x_0)\geq -\tau\chi\rho(x_0)$.
We now use the Monge-Amp\`ere equation connecting $\phi, T, \rho$ and $g$: 
$$
\rho(x_0)=g(T(x_0))\det(I-D^2\phi(x_0)).
$$
Since $I-D^2\phi(x_0)$ is a positive symmetric matrix, the arithmetic-geometric inequality provides the well-known inequality $\det(I-D^2\phi(x_0))^{1/d}\leq 1-\frac{\Delta\phi(x_0)}{d}$, hence
$$
\rho(x_0)\leq g(T(x_0))\left(1-\frac{\Delta\phi(x_0)}{d}\right)^d\leq ||g||_{L^\infty}\left(1+\tau\chi\frac{\rho(x_0)}{d}\right)^d.
$$
Setting $X:=\frac{\tau\chi}{d}||\rho||_{L^\infty}=\frac{\tau\chi}{d}\rho(x_0)$ and $Y=\frac{\tau\chi}{d}||g||_{L^\infty}$, we get
$$
Y\geq \frac{X}{(1+X)^d}.
$$
Looking at the behavior of $x\mapsto x/(1+x)^d$,
we observe that this function is maximal at $x=1/(d-1)$. By choosing the constant $c_0$ in the statement small enough, the value of $Y$ will be such that the level set $\{x\,:\,  x/(1+x)^d\leq Y\}$ will be composed of two intervals, one before the maximal point and one after. Yet, since we know $X\leq M\tau\chi/d=1/d<1/(d-1)$, the value of $X$ can only be in the first interval. If $c_0$, and hence $Y$, is small enough, this interval is a small neighborhood of the origin, on which we can use the inequality
$$
\frac{x}{(1+x)^d}\geq \frac{x}{1+\lambda dx},
$$
which is true for $\lambda>1$ and for small $x$ just by comparing the second derivatives of these two functions at $x=0$ (actually the values of the functions and of their first derivatives coincide at $x=0$). Hence, we get
$$
\frac{X}{1+\lambda dX}\leq Y,
$$
which is equivalent to
$$
\frac 1X+\lambda d\geq \frac 1Y
$$ 
and, substiting the values of $X$ and $Y$, we obtain exactly \eqref{1X1Y}. 

The last part of the statement is obtained by a standard approximation procedure. We take a sequence of smooth and bounded from below densities $g_n$ converging to $g$, and a sequence of functions $f_n$ defined via $f_n(t)=f(t)+n^{-1}t \log t$, and take the corresponding minimizers $\rho_n$. Up to subsequences, the minimizers $\rho_n$ converge to a minimizer $\rho$ for the Problem \eqref{min} with $g$ and $f$. This means that the same estimate is true for $\rho$, which is only one particular minimizer corresponding to $g$ and $f$, and not the unique one. This explains why the same estimate is not guaranteed for all the minimizers.
\end{proof}

The previous estimate can be iterated leading to an estimate on the evolution of the $L^\infty$-norm for this particular JKO scheme with constraints.

\begin{theorem}
Let $\rho_0\in L^\infty$ be a given probability density. Take $\chi t_0<||\rho_0||_{L^\infty}^{-1}$ and let $\lambda>1$ and $\ve_0>0$ be such that $||\rho_0||_{L^\infty}^{-1}-\chi \lambda t_0>\ve_0$. Fix $\tau<\ve_0 c_0$. Then there exists a sequence $\rho_{k}$ for $k=1,\dots,  \lfloor t_0/\tau \rfloor$ obtained by iteratedly solving \eqref{min}:
$$
\rho_{k+1}\in\argmin \left\{J(\rho)+\frac{W_2^2(\rho,\rho_k)}{\tau}\;:\,\rho\in\pical(\Omega), \,\rho\leq \frac{1}{\chi\tau}\right\}
$$
satisfying the $L^\infty$ estimate
\begin{equation}\label{iterated}
||\rho_k||_{L^\infty}\leq \frac{||\rho_0||_{L^\infty}}{1-\lambda k\tau ||\rho_0||_{L^\infty}}\leq \frac{1}{\ve_0} \quad \mbox{ or equivalently } \quad
||\rho_k||_{L^\infty}^{-1}\geq ||\rho_0||_{L^\infty}^{-1}-\lambda k\tau \geq \ve_0\,.
\end{equation}
\end{theorem}

\begin{proof}
We use the previous results, and prove by induction on $k$ that we may choose a minimizer $\rho_{k+1}$ satisfying the desired estimate. Note that $\rho_0$ satisfies the desired estimate. Now suppose we have $\rho_k$ satisfying the corresponding estimate, then Theorem \ref{min} can be applied with $g=\rho_k$ since, by assumption, we have $\tau<\ve_0 c_0$ and $||\rho_k||_{L^\infty}\leq \ve_0^{-1}$, which imply $\tau ||\rho_k||_{L^\infty}<c_0$. This provides the existence of a minimizer $\rho_{k+1}$ with 
$$
||\rho_{k+1}||_{L^\infty}^{-1}\geq ||\rho_k||_{L^\infty}^{-1}-\lambda\tau\chi.
$$
This, together with \eqref{iterated}, which is true by induction assumption and which reads $||\rho_k||_{L^\infty}^{-1}\geq ||\rho_0||_{L^\infty}^{-1}-\lambda k\tau\chi$, gives
$$||\rho_{k+1}||_{L^\infty}^{-1}\geq ||\rho_0||_{L^\infty}^{-1}-\lambda(k+1)\tau\chi,$$ 
which is the claim.
\end{proof}


\subsection{Variants: besides the case of Dirichlet conditions on bounded domains} 

We consider here easy variants of our result, which was stated on a bounded domain $\Omega$ with Dirichlet boundary conditions for the equation $-\Delta u=\rho$. Of course, the boundary conditions on the continuity equation are unchanged, they are always of no-flux type.\smallskip

\begin{itemize}
\item[1.] {\it Neumann or periodic boundary conditions.-} We can change the equation relating $u$ to $\rho$ into a PDE with Neumann boundary conditions, either considering
$$
-\Delta u=\rho-c\quad\mbox{ (with $c=1/|\Omega|$)}
$$
or $-\Delta u+u=\rho$, and imposing in both cases $\nabla u\cdot \mathbf n=0$ on $\partial\Omega$. The estimates that we can obtain are exactly the same and the arguments are very similar. The only differences are:
\begin{itemize}
\item When proving that we have $x_0\notin\partial\Omega$, we need to suppose that $\Omega$ is strictly convex, which guarantees $\nabla\phi(x_0)\cdot\mathbf n>0$ and allows to find a contradiction together with $\nabla u\cdot \mathbf n=0$. This assumption can be enforced by approximation, and the result stays true anyway on arbitrary convex domains.
\item It is no more true that we have $-\Delta u(x_0)=\rho(x_0)$, but we have in both cases $-\Delta u(x_0)\leq \rho(x_0)$. Indeed, in the first case we just use $c>0$, and in the second case $u\geq 0$ by the maximum principle.
\end{itemize}
The case of periodic boundary conditions, when $\Omega$ is the flat torus, is even simpler: the equation can be taken as $-\Delta u=\rho-c$
or
$-\Delta u+u=\rho$, and we do not need to distinguish the case $x_0\in \partial\Omega$. Of course, this requires to use regularity results for the optimal transport map on the torus, see \cite{cordero reg}.\smallskip

\item[2.] {\it The whole space case $\Omega=\R^d$.-} In this case define $u$ as
$$
u(x)=-U*\rho:=-\int U(x-y)\rho(y)\dd y
$$
where $U$ is the fundamental solution of the Laplacian, i.e. $\Delta U=\delta_0$, given by  $U(z)=\frac1{2\pi}\log(|z|)$ for $d=2$, and $U(z)=c_d|z|^{2-d}$ for $d>2$. In this way it is still true that we have $-\Delta u=\rho$, but we choose a precise representation formula instead of fixing the boundary conditions. Instead of the Dirichlet energy $-\int |\nabla u|^2$, we use 
$$
H(\rho):=-\int u(x)\rho(x)\dd x=\int\int \rho(x)\rho(y)U(x-y)\dd x \dd y\,.
$$
We use the fact that, in the space $L^1\cap L^\infty$, the energy $H(\rho)$ and the convolution $U*\rho$ satisfy the following estimates: 
\begin{itemize}
\item if $d>2$ then for each sequence $\rho_n$ bounded in $L^1\cap L^\infty$, $u_n:=-U*\rho_n$ is bounded in $L^\infty$ and hence $H(\rho_n)$ is bounded. This is a consequence of a simple splitting of the convolution in near and far field contributions. 

\item if $d=2$ then for each sequence $\rho_n$ bounded in $L^1\cap L^\infty$, $U *\rho_n$ is bounded from below by a uniform constant, again by splitting in near and far field contributions in the convolution and discarding the positive part of the logarithmic kernel. This also implies that $H(\rho_n)$ is bounded from below.
\end{itemize}
Moreover, we need to assume that our initial bounded probability density $g$ also has finite second moment. This is important to show the existence of minimizers, since one can take advantage of the $W_2^2(\rho,g)$ term to make appear the second moment of the minimizing sequence to help to control the negative part of the entropy in case $f(\rho)=\rho\log\rho$ or the negative part of the $f_n$ approximations for degenerate diffusions (in particular, for simplicity, we suppose in the unbounded case that $f$satisfies $f(s)\geq C(s \log (Cs))_-$ for some constant $C$). We can use the Carleman inequality \cite[Lemma 2.2]{BlaCalCar}, ot similar arguments (see \cite{BlaCalCar,BlaCarCar,CDFLS,San-moment}) which allows to show that there is no mass escaping to infinity and that the weak limit of the minimizing sequence as measures is a probability density. In fact, the convergence happens also weakly in $L^1(\R^d)$ by Dunford-Pettis theorem.

Again, in the space $L^1\cap L^\infty$, the energy $H(\rho)$ has good convergence properties: 
\begin{itemize}
\item if $d>2$ then for each sequence $\rho_n$ bounded in $L^1\cap L^\infty$ such that $\rho_n\deb \rho$, in the sense of weak convergence in duality with bounded and continuous functions on $\R^d$, then $H(\rho_n)\to H(\rho)$.

\item if $d=2$ then for each sequence $\rho_n$ bounded in $L^1\cap L^\infty$ such that the second moments are uniformly bounded in $n$ and $\rho_n\deb \rho$ weakly in $L^1(\R^d)$, then $H(\rho_n)\to H(\rho)$, see \cite[Lemma 3.1]{BlaCalCar} for the proof.
\end{itemize}

Then, the proof goes on in the same way, but in the approximation we will suppose that $g$ is a compactly supported density with uniformly convex and smooth support, smooth and bounded away from zero on its support (but discontinuous on the boundary of the support). This allows to apply the regularity theory for the Monge-Amp\`ere equation. More precisely, we need that the optimal transport map between two $C^{0,\alpha}$ probability densities with a source density being strictly positive on $\R^d$ and a target density being bounded away from zero on a bounded smooth convex domain is $C^{1,\alpha}(\R^d)$, and the Kantorovich potential is $C^{2,\alpha}(\R^d)$. This fact, for an unbounded source domain, but keeping boundedness and convexity of the target, is not easy to find in the literature, but can be proven by an easy adaptation of the arguments in \cite[Theorem 4.23, Remark 4.25]{Fig book}. By the way, the fact that the source domain is the whole space makes the proof even shorter. Moreover, in this case $\phi$ grows quadratically, and $u$ is bounded from above, which guarantees that a minimizer for $\phi-\tau\chi u$ exists. 

A final point to observe is that the results in \cite[Chapter 7]{OTAM} about the first variation of the $W_2^2$ cost are only proven in \cite{OTAM} in the case of a compact domain. Yet, generalizing them to the case where $g$ is compactly supported (and the comparison is done with measures $\tilde\rho$ which are compactly supported themselves) is an easy exercise. The rest of the proof is unchanged.
\end{itemize}


\section{Local existence of a solution, for arbitrary non-linearity and $L^\infty$ data}\label{3}

In this section we apply the previous estimates to prove the existence, local in time, of a weak solution of the parabolic-elliptic  Keller-Segel system \eqref{syst}. For simplicity we will start from the case of linear diffusion, and then turn to more arbitrary functions $f$. The proof will only underline the general strategy and the peculiarity of each case.

\

\noindent {\bf Linear diffusion.-} We consider the particular case
\begin{equation}\label{systlin}
\begin{cases}\partial_t \rho+\chi\nabla\cdot(\rho\nabla u)-\Delta\rho=0&\mbox{ in }[0,T]\times \Omega,\\
			-\Delta u=\rho &\mbox{ in $ \Omega$ for every $t\in [0,T]$},\\
			u=0&\mbox{ on $ \partial\Omega$ for every $t\in [0,T]$},\\
			(\rho\nabla u - \nabla\rho)\cdot \mathbf{n}=0&\mbox{ on $[0,T]\times \partial\Omega$},\\
			\rho(0,\cdot)=\rho_0&
			\end{cases}
			\end{equation}

\begin{theorem}\label{exist lin}
Suppose $\Omega$ is a bounded convex domain in $\R^d$. For any $\rho_0\in L^\infty(\Omega)$ and $T<(\chi ||\rho_0||_{L^\infty})^{-1}$, the system \eqref{systlin} admits at least a weak solution on $[0,T)\times\Omega$, and this solution is bounded for all $0\leq t<T$.
\end{theorem}
\begin{proof}
The construction follows the same scheme summarized in \cite{DMS} and \cite{GradFlowSurvey}.

By using the modified JKO scheme of the previous section, for each $\tau>0$ we are able to build a sequence $(\rho^{\tau}_k)_k$, iteratively solving \eqref{min}, where we choose $f(\rho)=\rho\log\rho$. the sequence is bounded in $L^\infty(\Omega)$ independently of $\tau$. We also define a sequence of velocities $\mathbf{v}^{\tau}_k=(id-\T)/\tau$, taking as $\T$ the optimal transport from $\rho^\tau _k$ to $\rho^{\tau}_{k-1}$. Now, we notice that the optimality conditions of \eqref{min}, since we know that in the end the constraint $\rho\leq M$ is not binding for $\tau$ small enough due to \eqref{iterated}, may be re-written without the pressure $p$, thus getting 
$$
\mathbf{v}^\tau_k=\frac{\nabla\phi}{\tau}=-\frac{\nabla\rho^\tau _k}{\rho^\tau _k}+\nabla u^\tau _k,
$$
where $u^\tau _k$ is the solution of $-\Delta u=\rho^\tau _k$ with Dirichlet boundary conditions on $ \partial\Omega$.

Then, we build two interpolating curves in the space of measures: 
\begin{itemize}
\item first we can define some piecewise constant curves, i.e. $\overline\rho^\tau_t:=\rho^{\tau}_{k+1}$ for $t\in ]k\tau,(k+1)\tau]$; associated with this curve we also define the velocities $\overline {\mathbf{v}^\tau_t}=\mathbf{v}^{\tau}_{k+1}$ for $t\in ]k\tau,(k+1)\tau]$ and the momentum variable $\overline E^\tau=\overline\rho^\tau \overline{ \mathbf{v}^\tau}$;
\item then, we can also consider  the densities $\widehat \rho^\tau_t$ that interpolate the discrete values $(\rho^\tau_k)_k$ along geodesics:
\begin{equation*}
\widehat \rho^\tau_t = \bigl( id - (k\tau-t) \mathbf{v}^\tau_k \bigr)_{\sharp} \rho^\tau_k,\;\mbox{ for }t\in](k-1)\tau,k\tau[;
\end{equation*}
the velocities $\widehat {\mathbf{v}^{\tau}_t}$ are defined so that $(\widehat \rho^\tau, \widehat{ \mathbf{v}^{\tau}})$ satisfy the continuity equation, taking 
$$\widehat {\mathbf{v}^{\tau}_t}=\mathbf{v}^\tau_t\circ \bigl( id - (k\tau-t) \mathbf{v}^\tau_k  \bigr)^{-1};$$ 
as before, we define: $\widehat E^{\tau}= \widehat \rho^\tau \widehat{ \mathbf{v}^{\tau}}$.
\end{itemize}

After these definitions we look for a priori bounds on the curves and the velocities that we defined.
Note that the optimality properties of the sequence $(\rho^\tau_k)_k$ give
$$
\frac{W_{2}^2(\rho^\tau_k, \rho^\tau_{k-1})}{\tau}\leq J(\rho^\tau_{k-1})-J(\rho^\tau_{k}).
$$
Hence, for $N=\lfloor T/\tau\rfloor$, we get
\begin{equation}\label{discreteH1}
\sum_{k=1}^N\tau\left(\frac{W_{2}(\rho^\tau_k, \rho^\tau_{k-1})}{\tau}\right)^2 \leq J(\rho_0)-J(\rho^\tau_N)\leq C,
\end{equation}
where the constant $C$ is independent of $\tau$, since $J(\rho_0)<+\infty$ and $J$ is bounded from below on the set of measures with density bounded by a given constant.
This is the discrete version of an $L^2$ estimate of the time derivative. As for $\widehat \rho^\tau_t$, it is an absolutely continuous curve in the Wasserstein space and its velocity on the time interval $[(k-1)\tau,k\tau]$ is given by the ratio $W_2(\rho^\tau_{k-1},\rho^\tau_k)/\tau$. Hence, the $L^2$ norm of its velocity on $[0,T]$ is given by
\begin{equation}\label{L2normv}
\int_0^T |(\widehat{\rho}^\tau)'|^2(t) dt=\sum_{k}\frac{W^2_{2}(\rho^\tau_k, \rho^\tau_{k-1})}{\tau}, 
\end{equation}
and, thanks to \eqref{discreteH1}, it admits a uniform bound independent of $\tau$. In our case, thanks to results on the continuity equation and the Wasserstein metric, this metric derivative is also equal to $||\widehat {\mathbf{v}^\tau_t}||_{L^2(\widehat\rho^\tau_t)}$. This gives compactness of the curves $\widehat\rho^\tau$, as well as an H\"older estimate on their variations in time (since $H^1\subset C^{0,1/2}$). The characterization of the velocities $\overline {\mathbf{v}^\tau}$ and $\widehat { \mathbf{v}^\tau}$ allows to deduce bounds on these vector fields from the bounds on $W_2(\rho^\tau_{k-1},\rho^\tau_k)/\tau$.

Considering all these facts, one obtains the following situation.
\begin{itemize}
\item The norm $\int ||\overline{\mathbf{v}_t^\tau}||^2_{L^2(\overline \rho^\tau_t)}dt$  is $\tau$-uniformly bounded.
\item In particular, the above $L^2$ bound is valid in $L^1$ as well, which implies that $\overline E^\tau$ is bounded in the space of measures over $[0,T]\times\Omega$.
\item The very same estimates are true for $\widehat {\mathbf{v}^\tau}$ and $\widehat E^\tau$.
\item The curves $\widehat\rho^\tau$ are bounded in $H^1([0,T],\mathbb W_2(\Omega))$ and hence compact in $C^0([0,T],\mathbb W_2(\Omega))$.
\item Up to a subsequence, one has $\widehat\rho^\tau\to\rho$, as $\tau\to 0$, uniformly according to the $W_2$ distance.
\item From the estimate $W_2(\overline\rho^\tau_t,\widehat\rho^\tau_t)\leq C\tau^{1/2}$ one gets that $\rho^\tau$ converges to the same limit $\rho$ in the same sense.
\item If we denote by $E$ a weak limit of $\widehat E^\tau$, since $(\widehat\rho^\tau,\widehat E^\tau)$ solves the continuity equation, by linearity, passing to the weak limit, also $(\rho,E)$ solves the same equation.
\item It is possible to prove, see \cite[Section 3.2, Step 1]{MRS10} or \cite[Chapter 8]{OTAM}, that the weak limits of $\widehat E^\tau$ and $\overline E^\tau$ are the same.
\item From the bounds in $L^2$ one gets that also the measure $E$ is absolutely continuous w.r.t. $\rho$ and has an $L^2$ density, so that we have for a.e. time $t$ a measure $E_t$ of the form $\rho_tv_t$.
\item It is only left to prove that one has $E_t=-\Delta\rho_t+\nabla\cdot(\rho_t \nabla u_t)$, where $u_t$ is the solution of $-\Delta u=\rho_t$ with Dirichlet boundary conditions on $ \partial\Omega$.\end{itemize}
In this case this last fact is easy to prove. Indeed, from $\bar\rho^\tau_t\deb\rho_t$ one immediately deduces $-\Delta\bar\rho^\tau_t\to-\Delta\rho_t$ (in the sense of distributions). The $L^\infty$ bound implies that the convergence $\bar\rho^\tau_t\deb\rho_t$ is weak in $L^p$, and hence $u(\bar\rho^\tau_t)\deb u(\rho_t)$ weakly in $W^{2,p}$. For $p>d$, this gives uniform convergence $\nabla u(\bar\rho^\tau_t)\to \nabla u(\rho_t)$ and guarantees that also the term $\nabla\cdot (\nabla u(\bar\rho^\tau_t)\bar\rho^\tau_t)$ converges in the sense of distribution (to $\nabla\cdot (\nabla u(\rho_t)\rho_t)$).
\end{proof}

\begin{remark}
The previous proof works in the whole space case almost entirely. Similarly to the proof of the existence of minimizers in the previous section, in order to have a bound from below on $J$, one applies Carleman estimates, for instance \cite[Lemma 2.2]{BlaCalCar}, to use second moments to estimate from below the entropy term $\int \rho\log\rho$. In order to do that, we can use part of the distance $W_2(\rho,\rho_0)$ as in \cite{BlaCalCar,BlaCarCar,CDFLS} or use a sublinear power of the left-hand side of \eqref{discreteH1}. For sharp lower bounds on the entropy in terms of moments, see \cite{San-moment}. Finally, passing to the limit in the equations only requires local convergences instead of global in the whole space, as this is enough for the equations to pass to the limit in the distributional sense.
\end{remark}

\noindent {\bf Non-linear diffusion.-} The problem becomes trickier in case of non-linear diffusions. However, a modification of the above proof, requiring the following lemma, inspired by the well-known Aubin-Lions compactness Lemma (see for instance \cite{Aubin}), allows to handle non-linear cases. The modification concerns the use of the Wasserstein distance and the simplification using the $L^\infty$ estimates, see \cite{MRS10} for similar computations.

\begin{lemma}\label{strong AL}
Suppose that $\mu^\tau$ is a sequence (indexed in $\tau$) of time-dependent probability densities such that
\begin{enumerate}
\item $||\mu^\tau_t||_{L^\infty}$ is bounded
\item $W_2(\mu^\tau_t,\mu^\tau_s)\leq C\sqrt{\tau}+\int_t^s g^\tau (r)dr$, where $g^\tau$ is bounded in $L^2([0,T])$ and $C$ is a constant.
\item For some strictly increasing and convex function $K:[0,+\infty)\to [0,+\infty)$ such that the function $s\mapsto s K^{-1}(s)$ is strictly convex, we have $||\nabla(K(\mu^\tau_t))||_{L^2(\Omega)}\leq Cg^\tau(t)$, for a given constant $C$ and the same function $g^\tau$ as above.
\end{enumerate}
Then, the sequence of  $\mu^\tau$ is compact in all the spaces $L^p$ with $p<\infty$ and in particular $\mu^\tau$ converges a.e. up to subsequences.
\end{lemma}

Before proving this lemma, we apply it to the case we are interested in.

\begin{theorem}
Suppose that $\Omega$ is a bounded convex domain in $\R^d$. Suppose that $f:\R\to\R$ is a continuous function on $[0,+\infty[$ and $C^2$ on $]0,+\infty[$, with $f''>0$. Then, for any $\rho_0\in L^\infty(\Omega)$ and $T<(\chi||\rho_0||_{L^\infty})^{-1}$, the system \eqref{syst} admits at least a weak solution on $[0,T]\times\Omega$, and this solution is bounded.
\end{theorem}
\begin{proof}
We follow the same proof as in Theorem \ref{exist lin}. The only difficulty left is to pass to the limit the term $\nabla\cdot(\bar\rho^\tau_t\nabla f'(\bar\rho^\tau_t)),$ as it has no more the linear form as it had before. Yet, this term can be written as $\Delta(\Psi(\bar\rho^\tau_t))$, where $\Psi(t)=tf'(t)-f(t)$. We know that we have weak convergence $\bar\rho^\tau_t\deb \rho_t$, and we just need to turn it into a.e. convergence to prove the desired convergence. This can be done by means of Lemma \ref{strong AL}, applied to $\mu^\tau=\bar\rho^\tau$. We need to check the assumptions of this lemma. Assumption (1) is guaranteed by our iterated $L^\infty$ estimate, and we call $M$ a constant such that $\bar\rho^\tau\leq M$ for all $\tau$. 

For assumptions (2) and (3), we define $g^\tau(t):=||\overline{\mathbf{v}_t^\tau}||_{L^2(\overline \rho^\tau_t)}+C$, for a suitable large constant $C$. Then, the integrability of $g^\tau$ comes from \eqref{L2normv}, and we have 
$$
W_2(\overline\rho^\tau_t,\overline\rho^\tau_s)\leq W_2(\overline\rho^\tau_t,\widehat\rho^\tau_t)+W_2(\overline\rho^\tau_s,\widehat\rho^\tau_s)+W_2(\widehat\rho^\tau_t,\widehat\rho^\tau_s)\leq C\tau^{1/2}+\int_t^s g^\tau(r)dr.
$$ 
In order to check the validity of assumption (3), we first choose, using the following Lemma \ref{Lemma h}, a function $k$ such that its antiderivative $K$, with $K(0)=0$, satisfies the requirements of assumption (3), and such that $k(s)\leq s^{1/2}f''(s)$ for all $s\in [0,M]$. This means that we have 
$$
||\nabla K(\mu^\tau_t)||_{L^2(\Omega)}^2= \int_\Omega k(\bar\rho^\tau)^{2}|\nabla \bar\rho^\tau|^2\leq \int_\Omega \bar\rho^\tau|f''( \bar\rho^\tau)|^2|\nabla \bar\rho^\tau|^2,
$$
but this last quantity is the squared norm in $L^2( \bar\rho^\tau)$ of $\nabla(f'( \bar\rho^\tau))$. From the optimality conditions defining $ \bar\rho^\tau$, we have 
$$
|\nabla f'( \bar\rho^\tau)|=\left|\nabla u+\frac{\nabla\phi}{\tau}\right|\leq C+\left|\frac{\nabla\phi}{\tau}\right|,
$$
where the constant $C$ is chosen so as to be larger than the $L^\infty$ norm of $\nabla u$ (which is bounded by the $L^\infty$ norm of $ \bar\rho^\tau$).
If we take the norms in $L^2( \bar\rho^\tau)$, we exactly get $||\nabla K(\mu^\tau_t)||_{L^2(\Omega)}\leq Cg^\tau(t)$. This concludes the proof.
\end{proof}

\begin{lemma}\label{Lemma h}
Given any strictly positive and continuous function $k_0:(0,M]\to (0,\infty)$, there exists a function $k:(0,M]\to (0,\infty)$ with $k\leq k_0$ such that
\begin{itemize}
\item $k$ is non-decreasing and $K$, its antiderivative defined through $K(s)=\int_0^s k$, is strictly increasing and convex;
\item the map $s\mapsto E(s):=sK^{-1}(s)$ is strictly convex.
\end{itemize}
\end{lemma}
\begin{proof}
In order to obtain a non-decreasing function $k\leq k_0$ it is enough to consider a strictly decreasing sequence $\ell_n\to 0$ with $\ell_0=M$ and define $k(s)=a_n:=\inf_{[\ell_{n+1},M]} h_0$ for all $s\in (\ell_{n+1},\ell_n]$. This value is strictly positive since $h_0$ is continuous and strictly positive. Automatically, if $k$ is strictly positive and non-decreasing, then $K$ is strictly increasing and convex. 

Then, we need to arrange this function so as to satisfy the second condition as well. For this we will choose a new function $k$, smaller than the one that we have just defined (i.e., smaller than $a_n$ on each interval $(\ell_{n+1},\ell_n]$), and locally Lipschitz. In this way $K$ will be locally $C^{1,1}$ and it is not difficult to check that the strict convexity of $E$ is satisfied whenever we have $K K''<2 |K'|^2$. In order to be convinced of this, just notice that $E'(s)=K^{-1}(s)+s/K'(K^{-1}(s))$; we want $E'$ to be strictly increasing, which is equivalent to $E'\circ K$ to be strictly increasing, hence we need to look at $s\mapsto s+K(s)/K'(s)$. If we differentiate it we obain $2-KK''/|K'|^2$.

We will define $k$ iteratively on each interval $(\ell_{n+1},\ell_n]$ in the following way: define three sequences of numbers $b_n, p_n, I_n>0$ such that
\begin{subequations} 
\begin{equation}\label{crazy1}
b_{n+1}=b_n (\ell_{n+1}/\ell_n)^{p_n}\,,
\end{equation}
\begin{equation}\label{crazy2}
b_n\leq a_n\,,
\end{equation}
\begin{equation}\label{crazy3}
I_np_n(\ell_n/\ell_{n+1})^{p_n}/(b_n\ell_{n+1})\leq 1\,,
\end{equation}
and
\begin{equation}\label{crazy4}
\ell_{n}b_{n}/(p_{n}+1)+b_{n}\ell_{n+1}(\ell_{n+1}/\ell_n)^{p_{n}}\leq I_{n-1}\,.
\end{equation}
\end{subequations}
In order to do so, fix $b_0=a_0, p_{-1}=1$ and $I_{-1}=2a_0M$. Then, if $(b_{n}, p_{n-1},I_{n-1})$ are given, we then choose $p_n$ large enough so that, when $b_{n+1}$ is defined using \eqref{crazy1}, both \eqref{crazy2} and \eqref{crazy4} are satisfied. This defines $p_n$ and $b_{n+1}$. We then choose $I_n$ small enough to that \eqref{crazy3} is satisfied, and iterate.

Once these numbers $(b_n, p_n, I_n)_n$ are fixed, we define $k(s)=b_n (s/\ell_n)^{p_n}\leq b_n\leq a_n$ for $s\in (\ell_{n+1},\ell_n]$. In this way for every $n$ we have $k(\ell_n)=b_n$. Moreover, the function $h$ defined in this way is strictly increasing, satisfies $k\leq k_0$, it is locally Lipschitz continuous, continuity at the points $\ell_n$ is guaranteed by \eqref{crazy1}, and we have, for $s\in (\ell_{n+1},\ell_n]$
\begin{align*}
K(s)&=\int_0^s k(r)\dd r\leq b_{n+2}\ell_{n+2}+\int_{\ell_{n+2}}^{\ell_{n+1}}k(r)\dd r+\int_{\ell_{n+1}}^{s}k(r)\dd r\\
&\leq b_{n+2}\ell_{n+2}+\int_0^{\ell_{n+1}} b_{n+1} (r/\ell_{n+1})^{p_{n+1}}\dd r+\int_0^{s} b_{n} (r/\ell_{n})^{p_{n}}\dd r\\
&= b_{n+1}(\ell_{n+2}/\ell_{n+1})^{p_{n+1}}\ell_{n+2}+\frac {\ell_{n+1}b_{n+1}}{p_{n+1}+1}+\frac {b_{n}s^{p_n+1}}{(p_{n}+1)(\ell_n)^{p_n}} \leq  I_n+\frac {b_{n}s^{p_n+1}}{(p_{n}+1)(\ell_n)^{p_n}}\,,
\end{align*}
where we used \eqref{crazy1} to pass from the second to the third line and \eqref{crazy4} for the last inequality. Now, in order to verify that we have $K K''<2 |K'|^2$, we just need to check that we have, for $s\in (\ell_{n+1},\ell_n]$,
$$
\left(I_n+\frac{1}{p_n+1}b_n \frac{s^{p_n+1}}{\ell_n^{p_n}}\right) p_nb_n \frac{s^{p_n-1}}{\ell_n^{p_n}}\left(b_n \frac{s^{p_n}}{\ell_n^{p_n}}\right)^{-2}<2,
$$
where we used the exact expression of $K'=k$ and $K''=k'$ on such an interval. Yet, the left hand side may be re-written as 
$$
\frac{p_n}{p_n+1}+I_np_n \frac{\ell_n^{p_n}}{b_ns^{p_n+1}} < 2,
$$
due to \eqref{crazy3}.
\end{proof}

We now prove Lemma \ref{strong AL}.
\begin{proof}
First of all we assume, up to subsequences, that the sequence $\mu^\tau$ converges to a function $\mu$, $K(\mu^\tau)$ to a function $A$ and $\mu^\tau K(\mu^\tau)$ to a function $B$, all the convergences being weak-* in $L^\infty([0,T]\times\Omega)$. We now fix two instants of time $a<b\in [0,T]$ and consider
$$
\fint_a^b\int_\Omega\mu^\tau K(\mu^\tau )\,\dd x \dd t\to \fint_a^b\int_\Omega B\,\dd x \dd t.
$$
Yet, we can also write
$$
\int_\Omega\mu^\tau_t K(\mu^\tau_t )\dd x=\int_\Omega K(\mu^\tau_t )\mu^\tau_a\dd x+\int_\Omega K(\mu^\tau_t )(\mu^\tau_t-\mu^\tau_a)\dd x.
$$
We use the $H^1$ bound on $K(\mu^\tau )$ to write
$$
\left|\int_\Omega K(\mu^\tau )(\mu^\tau_t-\mu^\tau_a)\dd x\right|\leq g^\tau(t)||\mu^\tau_t-\mu^\tau_a||_{X'},
$$
where $X$ is the Hilbert space of zero-mean $H^1$ functions endowed with the $L^2$ norm of the gradient, and $X'$ is its dual. It is well-known, see \cite{Loeper,MRS10} and \cite[Section 5.5]{OTAM}, that we have the interpolation estimate
$$
||\mu-\nu||_{X'}\leq \max\{||\mu||_{L^\infty},||\nu||_{L^\infty}\}^{1/2}W_2(\mu,\nu).
$$
Defining $G:[a,b]\to\R$ through $G(a)=0$ and $G'=g^\tau$, and using $W_2(\mu^\tau_t,\mu^\tau_a)\leq C\sqrt{\tau}+G(t)$, we get
$$\left|\fint_a^b\int_\Omega \mu^\tau K(\mu^\tau )\dd x \dd t- \fint_a^b\int_\Omega K(\mu^\tau_t )\mu^\tau_a\dd x \dd t\right|\leq \frac{C\sqrt{\tau}}{b-a}G(b)+\frac{1}{b-a}G^2(b).$$
We use $G^2(b)=|\int_a^b g^\tau(r)dr|^2\leq (b-a)\int_a^b g^\tau(r)^2dr$ and we define $\sigma^\tau$ to be the measure on $[0,T]$ with density $(g^\tau)^2$. Suppose (up to subsequences) that $\sigma^\tau$ weakly converges to a measure $\sigma$.
If we pass to the limit this inequality as $\tau\to 0$ we get
$$\left| \fint_a^b\int_\Omega B\,\dd x \dd t- \fint_a^b\int_\Omega \mu A\,\dd x \dd t\right|\leq \sigma([a,b]).$$
Note that in the above limit we could pass to the limit the term 
$$ \fint_a^b\int_\Omega K(\mu^\tau_t )\mu^\tau_a\dd x \dd t=\int_\Omega\mu^\tau_a\left(\fint_a^b K(\mu^\tau_t )\dd t\right)$$
since the function $x\mapsto \fint_a^b K(\mu^\tau_t(x) )\dd t$ is bounded in $H^1$ for fixed $a<b$. This means that it converges strongly and can be multiplied times the weak convergence of $\mu^\tau_a$, which converges to $\mu_a$ thanks to the continuity bound on $t\mapsto \mu^\tau_t$.

Now, for every $a$ which is a Lebesgue point for $t\mapsto (A,B)\in L^2(\Omega)^2$ and which is not an atom of $\sigma$, we get
$$\int_\Omega B(x,a)\dd x=\int_\Omega \mu_a A(x,a)\dd x\leq \int_\Omega  K^{-1}(A)A\,\dd x,$$
where we use the fact that $\mu^\tau=K^{-1}(K(\mu^\tau))$, and the function $s\mapsto K^{-1}(s)$ is concave on $\R_+$ (since $K$ is convex and increasing). This guarantees that the weak limit $\mu$ of $\mu^\tau$ must be smaller or equal than $K^{-1}$ of the weak limit of $K(\mu^\tau)$.

Then, we use the function $E(s)=sK^{-1}(s)$, and we have a sequence, $v^\tau:=E(\mu^\tau)$, of $L^\infty$ functions on $[0,T]\times\Omega$ such that $v^\tau\deb v$ and $\limsup_\tau\int E(v^\tau)\leq \int E(v)$ for a strictly convex function $E$. Because of strict convexity, we can write the inequality
$$E(s)\geq E(s_0)+E'(s_0)(s-s_0)+\omega(|s-s_0|),$$
where $\omega$ plays the role of a modulus of strict convexity and is a continuous function with $\omega(0)=0$ and $\omega(s)>0$ for $s>0$. Applying this inequality to $s=v^\tau$ and $s_0=v$, we obtain 
$$\limsup_\tau \int \omega(|v^\tau-v|) \leq \limsup_\tau \int E(v^\tau) - \int E(v) - \int E(v)(v^\tau-v)\leq 0,$$
where we used the weak convergence $v^\tau\deb v$ to handle the very last term.
This implies, up to subsequences, that $\omega(|v^\tau-v|)\to 0$ a.e. and hence that $v^\tau\to v$ a.e. In our case, composing with $E^{-1}$, this provides a.e. convergence for $\mu^\tau$ to $\mu$ and, if we also use the $L^\infty$ bound, $L^p$ convergence for every $p$. \end{proof}


\end{document}